\documentclass[11pt, notitlepage]{article}
\usepackage{amssymb,amsmath,comment}
\catcode`\@=11 \@addtoreset{equation}{section}
\def\thesection{\arabic{section}}

\def\theequation{\thesection.\arabic{equation}}
\catcode`\@=12
\usepackage{colortbl}
\usepackage[mathscr]{eucal}
\usepackage{epsf}
\usepackage{a4wide}

\newcommand{\e}{\epsilon}

\newcommand{\Om} {\Omega}

\newcommand{\la} {\lambda}
\newcommand{\La} {\Lambda}
\newcommand{\noi} {\noindent}

\newcommand{\uline} {\underline}
\newcommand{\oline} {\overline}
\newcommand{\mb} {\mathbb}

\usepackage[all]{xy}
\catcode`\@=11
\def\theequation{\@arabic{\c@section}.\@arabic{\c@equation}}
\catcode`\@=12

\newcommand{\QED}{\rule{2mm}{2mm}}

\newtheorem{Theorem}{Theorem}[section]
\newtheorem{Lemma}[Theorem]{Lemma}
\newtheorem{Proposition}[Theorem]{Proposition}

\newtheorem{Definition}[Theorem]{Definition}

\begin{document}

{\vspace{0.01in}}

\title
{ \sc On a class of weighted p-Laplace equation with singular nonlinearity}

\author{P. Garain\footnote{Department of Mathematics and Systems Analysis, Aalto University, Otakaari 1, 02150, Espoo, Finland.
 e-mail: pgarain92@gmail.com}~~ and
   ~~ Tuhina Mukherjee\footnote{T.I.F.R. Centre for Applicable Mathematics, Bengaluru-560065, India.
 e-mail: tulimukh@gmail.com} }

\date{}

\maketitle

\begin{abstract}
This article deals with the existence of the following quasilinear degenerate singular elliptic equation
\begin{equation*}
(P_\la)\left\{
\begin{split}
  -\text{div}(w(x)|\nabla u|^{p-2}\nabla u) &=  g_{\la}(u),\;u>0\; \text{in}\; \Om,\\
 u&=0 \; \text{on}\; \partial \Om,
\end{split}\right.
\end{equation*}
where $ \Om \subset \mb R^n$ is a smooth bounded domain, $n\geq 3$, $\la>0$, $p>1$ and $w$ is a Muckenhoupt weight. Using variational techniques, for $g_{\la}(u)= \la f(u)u^{-q}$ and certain assumptions on $f$, we show existence of a  solution to $(P_\la)$ for each $\la>0$. Moreover when $g_{\la}(u)= \la u^{-q}+ u^{r}$ we establish existence of atleast two solutions to $(P_\la)$ in a suitable range of the parameter $\la$. Here we assume $q\in (0,1)$ and $r \in (p-1,p^*_s-1)$.

\medskip

\noi \textbf{Key words: Weighted p-Laplacian, Singular Nonlinearity, Multiple Weak Solutions, Variational Method.}

\medskip

\noi \textit{2010 Mathematics Subject Classification:} 35R11, 35R09, 35A15.

\end{abstract}

\section{Introduction}
In this article, we are interested in the question of existence of weak solutions to the following singular weighted $p$-Laplace equation
\begin{equation*}
(P_\la)\left\{
\begin{split}
  -\Delta_{p,w} u &=  g_{\la}(u),\;u>0\; \text{in}\; \Om,\\
 u&=0 \; \text{on}\; \partial \Om,
\end{split}\right.
\end{equation*}
where $ \Om \subset \mb R^n$ is a smooth bounded domain, $n\geq 3$, $\la>0$ and $p>1$.
We consider the nonlinearity $g_{\la}$ of the following two types:\begin{enumerate}
\item[Case (I)] $g_{\la}(u) = \la f(u)u^{-q}$ where $q\in (0,1)$ and $f:[0,\infty)\to \mb R$ satisfies
\begin{enumerate}
\item[(f1)] $f(0)>0$ such that $f$ is non decreasing and satisfies the following hypothesis:
$$
\lim_{t\to\infty}\frac{f(t)}{t^{q+p-1}}=0\text{ and } \lim_{t\to 0}\frac{f(t)}{t^q}=\infty.
$$
\end{enumerate}

\item[Case (II)] $g_{\la}(u)= \la u^{-q}+ u^{r}$ where $q\in (0,1)$, $r\in(p-1,p^*_s-1)$. Here $p^*_s= \frac{np_s}{n-p_s}$ for   $1\leq p_s <n$ where $p_s=\frac{ps}{s+1}$ and $s\in [\frac{1}{p-1},\infty) \cap (\frac{n}{p},\infty)$.
\end{enumerate}

We observe that in both the cases, Case (I) and Case (II), $g_\la$ is singular in the sense that
$$
\lim_{t\to 0^{+}}\,g_{\la}(t)=+\infty.
$$
Here
$$
\Delta_{p,w}u :=\text{div}(w(x)|\nabla u|^{p-2}\nabla u)
$$
is the weighted $p$-Laplace operator for some weight function $w$. When $w\equiv 1$, $\Delta_{p,w}=\Delta_{p}u$ (which is the usual $p$-Laplace operator) which further reduces to the classical Laplace operator '$\Delta$' for $p=2$.

In this article, we discuss the existence of weak solutions to the problem $(P_\la)$ depending on the range of $\la$. The study of singular elliptic problems has been a topic of considerable attention throughout the last three decade and there is a colossal amount of work done in this direction. Now, we state some known results in this direction which are essential to understand the difficulties and the framework of our problem. The following quasilinear singular problem has been investigated in quite a large number of papers:
\begin{equation}\label{plap}
\left\{
\begin{split}
  -\Delta_{p} u &=  \la\frac{h(x,u)}{u^q}+\mu u^{r},\;u>0\; \text{in}\; \Om,\\
 u&=0 \; \text{on}\; \partial \Om.
\end{split}\right.
\end{equation}
When $p=2$, $\la>0$ and $\mu=0$ (that is the purely singular case) Crandall, Rabinowitz and Tartar \cite{CRT} proved the existence of a unique classical solution $u_\la\in C^2(\Om)\cap C(\overline{\Om})$ of the problem \eqref{plap} for any $q>0$ and $h(x,u)=h(x)$ being nonnegative and bounded in $\Om$. For the same problem, existence of a weak solution in $H_0^1(\Om)$ was proved by Lazer-Mckenna \cite{LMckena} when $0<q<3$. Boccardo-Orsina \cite{BL} investigated the following purely singular problem in case of arbitrary $q>0$ with a weight function
\begin{equation}\label{2lap}
\left\{
\begin{split}
  -\text{div}(w(x)\nabla u) &=  \frac{h(x)}{u^q},\;u>0\; \text{in}\; \Om,\\
 u&=0 \; \text{on}\; \partial \Om,
\end{split}\right.
\end{equation}
where $w(x)$ satisfies 
\begin{equation}\label{wgtcondition}
w(x)\eta\cdot\eta\geq \alpha|\eta|^2,\,|w(x)|\leq\beta
\end{equation}
for some positive constants $\alpha,\beta$ and $\eta\in \mb R^n$. They proved existence of a weak solution $u\in H_{0}^1(\Om)$ for $0<q\leq 1$ and $u\in H^{1}_{loc}(\Om)$ for $q>1$ such that $u^\frac{q+1}{2}\in H_{0}^1(\Om)$. Recently, Canino-Sciunzi-Trombetta \cite{Canino1} generalized the problem \eqref{2lap} for the $p$-Laplace operator and obtained both the existence and uniqueness of weak solutions. This problem in the weighted case has been studied in \cite{PG}. When $h(x,u)=h(u)\geq 0$ and $\mu =0$, \eqref{plap} was investigated by Ko-Lee-Shivaji in \cite{KLS} for any $p>1$ and $0<q<1$ in a certain range of $\la$.

Concerning the perturbed case i.e. $\mu>0$, Haitao \cite{haitoh} proved the existence of at least two solutions for \eqref{plap}  using Perron's method and assuming $p=2,0<q<1<r\leq\frac{n+2}{n-2}$, $h \equiv 1$ and $\mu=1$ for certain range of $\la>0$. Simultaneously, using the Nehari manifold approach, Hirano-Saccon-Shioji \cite{hirano1} proved multiplicity result for the problem \eqref{plap} with the assumptions as in \cite{haitoh} in a certain range of $\la>0$. In a natural way of extension, the weighted Laplace equation with singular nonlinearity and a perturbation term that is 
\begin{equation}\label{3lap}
\left\{
\begin{split}
  -\text{div}(w(x)\nabla u) &= \frac{\la}{u^q}+u^r,\;u>0\; \text{in}\; \Om,\\
 u&=0 \; \text{on}\; \partial \Om,
\end{split}\right.
\end{equation}
where $w(x)$ satisfies the hypothesis \eqref{wgtcondition}, existence of a weak solution for  \eqref{3lap} has been proved in \cite{Boc3} for any $q>0$ and multiplicity result has been established in \cite{arcoya} under the restriction $0<q<1.$ Later on, authors in \cite{DArcoya} obtained multiplicity results for the problem \eqref{3lap} for any $q>0$ when $w\equiv 1$ and $\la$ lies in an appropriate range. Moving on to the quasilinear case, multiplicity result for the following singular $p$-Laplace equation
\begin{equation}\label{4lap}
\left\{
\begin{split}
  -\Delta_{p}u &= \frac{\la}{u^q}+u^r,\;u>0\; \text{in}\; \Om,\\
 u&=0 \; \text{on}\; \partial \Om,
\end{split}\right.
\end{equation}
has been established by Giacomoni-Schindler-Tak\'{a}\v{c} in \cite{Giacomoni} assuming $0<q<1$ and $p-1<r\leq p^*{-1}$ for certain $\la>0.$ The problem \eqref{4lap} with $q\geq 1$ has been recently settled in \cite{KBal}. We refer to \cite{book-radu} for a comprehensive list of bibliography related to semilinear singular Dirichlet problems.

 From the above literature it is clear that singular problems has been almost settled for the $p$-Laplace operator which is degenerate for $p>2$ and singular for $1<p<2$ at the critical points (see \cite{PLin}). Such degeneracy behavior of the operator also depends on the weight function '$w$' as in our case for the operator $\Delta_{p,w}$, even in the case $p=2$ (since $\Delta_{2,w}=w(x)\Delta u+\nabla w\cdot\nabla u$). Motivated by the singular problems with weighted p-Laplace operator studied in literature, it is natural to ask the question of existence when $w$ violates the hypothesis \eqref{wgtcondition}, specially if $w\to 0$ or $w\to\infty$ (e.g.,\,$w(x)=|x|^\alpha$) which captures the degenerate behavior of $\Delta_{p,w}$. For more details on such operators, we refer to \cite{Drabek, EFabes, Juh}. 
 
 The problem $(P_\la)$ for Case (I) when $w(x) \equiv 1$ has been studied by Ko, Lee and Shivaji in \cite{KLS} and  $(P_\la)$ for Case (II) has been studied by Arcoya and Bocardo in \cite{arcoya} when $p=2$ and $w(x)$ satisfying \eqref{wgtcondition}. Our main focus in this article is to provide a class of weights $w$ which assures the existence of weak solutions to the problem $(P_{\la})$ in both Case (I) and Case (II).
 We start with choosing the weight function in the class of Muckenhoupt weights $A_p$ (refer to section 2 for definition and see \cite{Muc} for more details). Then we define a subclass $A_s$ of $A_p$ which ensures some crucial embedding results (see section 2). When $w\in A_s$, Garain in \cite{PG} proved existence of solution to
\[-\Delta_{p,w}u = u^{-q},\; u>0\; \text{in}\; \Om,\; u=0\;\text{on}\; \partial \Om\]
such that $u(x)\geq c_K>0$ when $x \in K \subset \subset \Om$. We use this property of solutions to the purely singular problem with $-\Delta_{p,w}$ very efficiently to construct sub solution for $(P_\la)$. Then using Perron's idea, we show that $(P_\la)$ in Case (I) possesses a bounded weak solution. To prove the multiplicity result, later we consider a parameter dependent perturbed problem $(P_\la)$ in Case (II). Here, we consider an approximated problem $(P_{\la,\e})$ and showed existence of two weak solutions $\zeta_\e,\nu_\e$ to it using the Mountain pass Lemma. Next, we lead to passing the limit as $\e \to 0$ on $\{\zeta_\e\}$ and $\{\nu_\e\}$ which contributes two weak solutions to $(P_\la)$ in Case (II). The key point of this article is that we do not require any regularity results and proved our main theorems using purely variational techniques although the weight $w$ here can be possibly singular. The results proved here are completely new concerning the singular problem with weighted $p$-Laplace operator.\\

We have divided our paper into four sections: Section 2 contains the variational framework and preliminaries. Section 3 contains the main result related to $(P_\la)$ in Case (I) and Section 4 contains the multiplicity result for $(P_\la)$ in Case (II).

\section{Variational Framework}
We begin this section by briefly introducing the weighted Sobolev space corresponding to the Muckenhoupt weight, for more details refer to \cite{Drabek, EFabes, Juh, Tero, Muc}.

\begin{Definition}{(Muckenhoupt Weight)}
Let $w$ be a locally integrable function in $\mathbb{R}^n$ such that $0<w<\infty$ a.e. in $\mathbb{R}^n$. Then we say that $w$ belong to the Muckenhoupt class $A_p$, $1<p<\infty$ if there exists a positive constant $c_{p,w}$ (called the $A_p$ constant of $w$) depending only on $p$ and $w$ such that for all balls $B$ in $\mathbb{R}^n$,
$$
\left(\frac{1}{|B|}\int_{B}w \,dx\right)\left(\frac{1}{|B|}\int_{B}w^{-\frac{1}{p-1}} \,dx\right)^{p-1}\leq c_{p,w}.
$$
\end{Definition}

\textbf{Example} $w(x) = |x|^\alpha \in A_p$ if and only if $-n < \alpha < n(p-1)$ for any $1<p<\infty$, see \cite{Juh, Tero}.

\begin{Definition}(Weighted Sobolev Space)
For any $w \in A_p$, we define the weighted Sobolev space $W^{1,p}(\Omega,w)$ by
$$
W^{1,p}(\Omega,w)=\{u:\Omega\to\mathbb{R}\text{ measurable }:\|u\|_{1,p,w}<\infty\},
$$
with respect to the norm $\|.\|_{1,p,w}$ defined by
\begin{equation}\label{norm1}
\|u\|_{1,p,w} = \left(\int_{\Omega}|u(x)|^{p} w(x)\,dx\right)^\frac{1}{p} + \left(\int_{\Omega}|\nabla u|^{p} w(x)\,dx\right)^\frac{1}{p}.
\end{equation}
\end{Definition}
Also we define the space $W_{0}^{1,p}(\Omega,w)=\overline{(C_c^\infty(\Omega),\|\cdot\|_{1,p,w})}$ and denote it by $X$.
\begin{Lemma}(Poincar\'e inequality \cite{Juh})\label{Poincare inequality}
For any $w \in A_p$, we have
\begin{equation}
\int_{\Omega}|\phi|^{p} w(x) \,dx \leq C\,(\mbox{diam}\;\Omega)^p\int_{\Omega}|\nabla\phi|^{p} w(x) \,dx,\;\forall\;\phi\in C_{c}^\infty(\Omega),
\end{equation}
for some constant $C>0$ independent of $\phi$.
\end{Lemma}

Using Lemma \ref{Poincare inequality}, an equivalent norm to $(\ref{norm1})$ on the space $X$ can be defined by
\begin{equation}\label{norm2}
\|u\|=\left(\int_{\Omega}|\nabla u(x)|^pw(x)dx\right)^\frac{1}{p}.
\end{equation}

\noi \textbf{A subclass of $A_p$:} Let us define a subclass of $A_p$ by
$$
A_s = \left\{w\in A_p: w^{-s}\in L^{1}(\Omega)\,\,\text{for some}\,\,s\in[\frac{1}{p-1},\infty)\cap(\frac{n}{p},\infty)\right\}.
$$
For example, $w(x)=|x|^\alpha\in A_s$ for any $-\frac{n}{s}<\alpha<\frac{n}{s}$, provided $1<p<n$.

\begin{Lemma}\label{alg-ineq}{(Algebraic Inequality, Lemma A.0.5 \cite{Peral})}
For any $x,y\in\mb{R}^n$, one has
\begin{equation*}
\langle|x|^{p-2}x-|y|^{p-2}y,x-y\rangle\geq \left\{
\begin{split}
&  c_p|x-y|^p, \text{ if } p\geq 2,\\
& c_p\frac{|x-y|^2}{(|x|+|y|)^{2-p}}, \text{ if } 1<p<2,
\end{split}\right\}
\end{equation*}
where $\langle.,.\rangle$ denotes the standard inner product in $\mb{R}^n$.
\end{Lemma}

\begin{Lemma}\label{embedding}(Embedding)
For any $w \in A_s$, we have the following continuous inclusion map
\[
    X\hookrightarrow W_{0}^{1,p_s}(\Omega)\hookrightarrow
\begin{cases}
    L^q(\Omega),& \text{for } p_s\leq q\leq p_s^{*}, \text{in case of } 1\leq p_s<n, \\
    L^q(\Omega),& \text{for } 1\leq q< \infty, \text{in case of } p_s=n, \\
    C(\overline{\Omega}),& \text{in case of } p_s>n,
\end{cases}
\]
where $p_s = \frac{ps}{s+1}$ and $p_s^*=\frac{np_s}{n-p_s}$ is the critical Sobolev exponent.

Moreover, the above embeddings are compact except for $q=p_s^{*}$ in case of $1\leq p_s<n$.
\end{Lemma}

\begin{proof}
For proof refer to Theorem 2.15 of \cite{PG}.\hfill{\QED}
\end{proof}

\begin{Definition}{(Weighted Morrey space)}
Let $1<p<\infty$, $t>0$ and $w\in A_p$. Then we say that $u$ belong to the weighted Morrey space $L^{p,t}(\Omega,w)$, if $u\in L^p(\Omega,w)$, where
$$
L^p(\Om,w)=\left\{u:\Om\to\mb{R}\text{ measurable }:\int_{\Om}|u|^p\,w(x)\,dx<\infty\right\}
$$
 and
$$
\|u\|_{L^{p,t}(\Omega,w)}:=\sup_{{x\in\Omega},{{0<r<d_0}}} \left(\frac{r^t}{\mu(\Omega\cap B(x,r))}\int_{\Omega\cap B(x,r)}|u(y)|^pw(y)\,dy\right)^\frac{1}{p}<\infty,
$$
where $d_0=\text{diam}(\Omega)$ and $\mu(\Omega\cap B(x,r))=\int_{\Omega\cap B(x,r)}w(x)\,dx,$ and $B(x,r)$ denotes the ball with center $x$ and radius $r.$
\end{Definition}

\textbf{Assumption on the weight function '$w$':} Throughout the paper, we assume the following
\begin{itemize}
\item for $p_s>n$, the weight function $w\in A_s$ and
\item for $1\leq p_s\leq n$, the weight function $w\in A_s$ such that
$$\frac{1}{w}\in L^{q,pn-\alpha q(p-1)}(\Omega,w),$$
for some $q>n$ and $0<\alpha<\text{min}\{1,\frac{pn}{q(p-1)}\}$.
\end{itemize}

\begin{Lemma}\label{Uniform}
Let $u\in X$ be positive which solves the equation $-\Delta_{p,w}u=g$ for some $g\in L^\infty(\Omega)$. Then $u\geq c_K>0$ for every $K\subset\subset\Omega$.
\end{Lemma}

\begin{proof}
Let $p_s>n$, then the result follows by Lemma \ref{embedding}. If $1\leq p_s\leq n$, then arguing similarly as in Theorem 3.13 of \cite{PG} we get $u\in L^\infty$. Now applying Theorem 1.3 of \cite{Peng} we get the desired result. \hfill{\QED}
\end{proof}

\begin{Definition}
We say that $u\in X$ is a weak solution of $(P_\la)$ if $u>0$ in $\Om$ and for all $\phi\in C_c^{\infty}(\Om)$, one has
\begin{equation}\label{weak-sol}
\int_{\Om}w(x)|\nabla u|^{p-2}\nabla u\cdot\nabla\phi\,dx=\int_{\Om}g_{\la}(u)\phi\,dx.
\end{equation}
\end{Definition}
Moreover we say a function $u\in X$ to be a subsolution (or supersolution) of $(P_\la)$ if 
\begin{equation}\label{subsup-sol}
\int_{\Om}w(x)|\nabla u|^{p-2}\nabla u\cdot\nabla\phi\,dx\leq (\text{or }\geq)\int_{\Om}g_{\la}(u)\phi\,dx
\end{equation}
for every $0\leq \phi\in C_c^\infty(\Om)$.

Throughout the article we denote by $X_{+}=\{u\in X:u\geq 0 \text{ a.e. in }\Om\}$, $v^{+}(x)=\text{max}\{v(x),0\},\,v^{-}(x)=\text{max}\{-v(x),0\},$ $|S|=\text{Lebesgue measure of }S,\,p'=\frac{p}{p-1}$ for $p>1.$ Then we have the following property of weak solutions.
\begin{Lemma}\label{testfn}
\eqref{weak-sol} holds for every $\phi\in X$.
\end{Lemma}
\begin{proof}
Following the proof of Lemma A.1 of \cite{hirano1}, we get for any $v\in X_{+}$, there exists a sequence $\{v_n\}\in X$ such that each $v_n$ has a compact support in $\Omega$, $0\leq v_1\leq v_2\leq\ldots$ and $\{v_n\}$ converges strongly to $v$ in $X$. Now arguing similarly as in Lemma 9 of \cite{hirano1} we get the result.
\end{proof} \hfill{\QED}

Our main results related to problem $(P_\la)$ reads as:
\begin{Theorem}\label{MT1}
There exists a weak solution to $(P_\la)$ for every $\la>0$ under the assumption $(f1)$ in Case (I).
\end{Theorem}

\begin{Theorem}\label{MT3}
There exists a $\Lambda>0$ such that when $\la \in (0, \Lambda)$, $(P_{\la})$ admits at least two weak solutions in Case (II).
\end{Theorem}

\section{Existence result in Case (I)}
In this section, we head towards proving our first main result that is Theorem \ref{MT1} using the method of sub and supersolution. Let us first define our energy functional $E_\la: X \to \mb R \cup \{\pm\infty\}$ corresponding to $(P_\la)$ as
\[E_\la(u)= \frac{1}{p}\int_{\Om} w(x)|\nabla u|^p~dx - \la \int_{\Om}F(u)~dx  \]
where
\begin{equation*}
F(t)=\left\{
\begin{split}
&\int_0^t \frac{f(\tau)}{\tau^q}~d\tau,\; \text{if}\; t>0,\\
&0, \; \text{if}\; t\leq 0.
\end{split}\right.
\end{equation*}
Then the following Lemma is a crucial result to obtain the existence of solution and we follow \cite{haitoh}.

\begin{Lemma}\label{Subsuplemma}
Let $\underline{u}, \overline{u} \in X \cap  L^\infty(\Om)$ be sub and supersolution of $(P_{\la})$ respectively such that $0\leq \underline{u} \leq \overline{u}$ and $\uline{u}\geq c_K >0$ for every $K \subset \subset \Om$, for some constant $c_K$. Then there exists a weak solution $u\in X \cap L^\infty(\Om)$ of $(P_\la)$ satisfying $\uline{u}\leq u\leq \oline{u}$ in $\Om$.
\end{Lemma}

\begin{proof}
Consider the set
$$
M=\{v\in X:\uline{u}\leq v\leq \oline{u}\text{ in }\Om\}.
$$
By the given condition $\uline{u}\leq \oline{u}$ in $\Om$, so $M\neq \emptyset$. Also it is standard to check that $M$ is closed and convex. \\
\textbf{Claim (1):} $E_\la$ is weakly sequentially lower semicontinuous on $M$.\\
To show this, consider a sequence $\{v_k\} \subset M$ such that $v_k \rightharpoonup v$ weakly in X. Then using (f1) we have
\[F(v_k) \leq \int_0^{\oline{u}} \frac{f(\tau)}{\tau^q}~d\tau \leq \frac{f(\|\oline{u}\|_\infty)}{(1-q)}\|\oline{u}\|_\infty^{1-q}.\]
Therefore from Lebesgue Dominated Convergence theorem and weak lower semicontinuity of norms, the claim follows. So there exists a minimizer $u \in M$ of $E_\la$ that is $E_\la(u)= \inf\limits_{v\in M}E_\la(v)$.\\
\textbf{Claim (2):} $u$ is a weak solution of $(P_\la)$.\\
Let $\phi \in C_c^\infty(\Om)$ and $\e>0$ then we define
\begin{equation*}
\eta_{\e}=\left\{
\begin{split}
&\oline{u} \;\;\;\text{if} \; u+\e \phi \geq \oline{u}\\
&u+\e \phi  \;\;\;\text{if}\; \uline{u}\leq u+\e \phi \leq \oline{u}\\
&\uline{u} \;\;\;\text{if} \; u+\e \phi \leq \uline{u}.
\end{split}\right.
\end{equation*}
Observe that $\eta_{\e}=u+\e\phi-\phi^{\e}+\phi_{\e}\in M.$
For notational convenience, let us denote $\phi^{\e}= (u+\e \phi -\oline{u})^+$ and $\phi_{\e}= (u+\e \phi -\uline{u})^-$. Now from definition of $u$, we have
\begin{equation}\label{eq1}
\begin{split}
0 & \leq \lim_{t \to 0} \frac{E_\la(u+t(\eta_{\e}-u))- E_\la(u)}{t}\\
& = \lim_{t \to 0} \frac{1}{p} \frac{\displaystyle\int_{\Omega} w(x)(|\nabla u+t\nabla(\eta_\e-u)|^p-|\nabla u|^p)~dx }{t} - \la\lim_{t \to 0} \frac{\displaystyle \int_\Om (F(u+t(\eta_\e-u))- F(u)) ~dx}{t}\\
&= I_1-\la I_2 \;\text{(say)}.
\end{split}
\end{equation}
 It is easy to see that
 \[I_1 = \int_\Om w(x)|\nabla u|^{p-2}\nabla u. \nabla (\eta_\e-u)~dx.\]
   Next, we consider the quantity $I_2$ and get that
   \[I_2 = \lim_{t\to 0} \int_\Om \frac{(\eta_\e-u)f(u+\theta t (\eta_\e-u))}{(u+\theta t (\eta_\e-u))^q}~dx,\; \text{for some}\; \theta \in(0,1).\]
If $(\eta_\e-u)\geq 0$ then from Fatou's Lemma, it follows that
\[I_2 \geq \int_\Om \frac{(\eta_\e-u)f(u)}{u^q}~dx.\]
 Otherwise if $(\eta_\e-u) <0$ then since $(\eta_\e-u)\geq \e\phi$, so $\phi \leq 0$. Hence in this case
 \[\left| \frac{(\eta_\e-u)f(u+\theta t (\eta_\e-u))}{(u+\theta t (\eta_\e-u))^q} \right|\leq \frac{-(\eta_\e-u)f(||\oline{u}||_{\infty})}{\uline{u}^q}\leq \frac{-\e\phi f(||\oline{u}||_\infty)}{\uline{u}^q}\in L^1(\Om)\]
 since $\phi\in C_c^{\infty}(\Om)$ and $\uline{u}\geq c_{K}>0,$ whenever $K\subset\subset\Om$.

By Lebesgue Dominated Convergence theorem,
$$
\la I_2=\la\int_{\Om}\frac{(\eta_\e-u)f(u)}{u^q}\,dx.
$$
Using these in \eqref{eq1} we obtain
\begin{equation}\label{eq2}
\begin{split}
&0\leq  \int_\Om  w(x)|\nabla u|^{p-2}\nabla u. \nabla (\eta_\e-u)~dx - \la \int_\Om \frac{(\eta_\e-u)f(u)}{u^q}~dx\\
& \implies \frac{1}{\e}(Q^\e-Q_\e)\leq \int_\Om  w(x)|\nabla u|^{p-2}\nabla u. \nabla \phi~dx - \la \int_\Om \frac{f(u)}{u^q}\phi~dx
\end{split}
\end{equation}
where
\[Q^\e= \int_\Om  w(x)|\nabla u|^{p-2}\nabla u. \nabla \phi^\e~dx - \la \int_\Om \frac{f(u)}{u^q}\phi^\e~dx \]
\[\text{and}\;Q_\e= \int_\Om  w(x)|\nabla u|^{p-2}\nabla u. \nabla \phi_\e~dx - \la \int_\Om \frac{f(u)}{u^q}\phi_\e~dx .\]
Now we estimate $Q^\e $ and $Q_\e$ separately. So consider
\begin{align*}
\frac{1}{\e}Q^\e &\geq \frac{1}{\e} \int_\Om w(x)(|\nabla u|^{p-2}\nabla u- |\nabla \oline{u}|^{p-2}\nabla \oline{u}). \nabla \phi^\e~dx + \frac{\la}{\e}\int_\Om \frac{f(\oline{u})}{\oline{u}^q}\phi^\e~dx-\frac{\la}{\e}\int_\Om \frac{f({u})}{{u}^q}\phi^\e~dx\\
&=\frac{1}{\e}\int_{\Om^{\e}}w(x)(|\nabla u|^{p-2}\nabla u-|\nabla\oline{u}|^{p-2}\nabla\oline{u}).\nabla(u-\oline{u})\,dx\\
& \quad \quad +\int_{\Om^{\e}}w(x)(|\nabla u|^{p-2}\nabla u-|\nabla\oline{u}|^{p-2}\nabla\oline{u}).\nabla\phi\,dx
+\frac{\la}{\e}\int_{\Om}\left(\frac{{f({u})}}{\oline{u}^q}-\frac{f(u)}{u^q}\right)\phi^{\e}\,dx\\
&\geq \int_{\Om^{\e}}w(x)(|\nabla u|^{p-2}\nabla u-|\nabla\oline{u}|^{p-2}\nabla\oline{u}).\nabla\phi\,dx+\frac{\la}{\e}\int_{\Om^{\e}}{f(u)}\left(\frac{1}{\oline{u}^q}-\frac{1}{u^q}\right)(u-\oline{u})\,dx\\
&\quad \quad+\la\int_{\Om^{\e}}f(u)\left(\frac{1}{\oline{u}^q}-\frac{1}{u^q}\right)\phi\,dx\\
&\geq O(1)
\end{align*}
using Lemma \ref{alg-ineq}, $\oline{u}$ is a supersolution of $(P_{\la})$, $u\leq\oline{u}$ and $\displaystyle\int_{\Om^{\e}}\frac{f(u)}{\oline{u}^q}\phi\,dx\leq\frac{f(||\oline{u}||_{\infty})}{c_K^q}||\phi||_{\infty}<+\infty$, where $\Om^{\e}= \text{supp}\;\phi^\e$. Next we consider
\begin{align*}
\frac{1}{\e}Q_\e &\leq -\frac{1}{\e}\int_{\Om_{\e}} w(x)|\nabla u|^{p-2}\nabla u. \nabla(u+\e \phi -\uline{u})~dx+\frac{1}{\e}\int_{\Om_{\e}}w(x)|\nabla\uline{u}|^{p-2}\nabla\uline{u}.\nabla(u+\e\phi-\uline{u})\,dx\\
&\quad \quad+\frac{\la}{\e}\int_{\Om}\frac{f(\uline{u})}{\uline{u}^q}\phi_{\e}\,dx-\frac{\la}{\e}\int_{\Om}\frac{f(u)}{u^q}\phi_{\e}\,dx\\
&\leq \int_{\Om_{\e}}w(x)(|\nabla\uline{u}|^{p-2}\nabla\uline{u}-|\nabla u|^{p-2}\nabla u).\nabla\phi\,dx-\frac{\la}{\e}\int_{\Om_{\e}}f(u)\left(\frac{1}{\uline{u}^q}-\frac{1}{u^q}\right)(u-\uline{u})\,dx\\
& \quad \quad-\la\int_{\Om_{\e}}f(u)\left(\frac{1}{\uline{u}^q}-\frac{1}{u^q}\right)\phi\,dx\\
&\leq O(1)
\end{align*}
using Lemma \ref{alg-ineq}, $\uline{u}$ is a subsolution of $(P_\la)$, $u \geq \uline{u}$ and $\displaystyle \int_{\Om_{\e}}f(u)\left(\frac{1}{\uline{u}^q}-\frac{1}{u^q}\right)\phi\,dx \leq \frac{2f(\|\oline{u}\|_\infty)}{c_K^q}\|\phi\|_\infty<+\infty$
.\end{proof}
Putting these in \eqref{eq2} we obtain
\[0 \leq \int_\Om  w(x)|\nabla u|^{p-2}\nabla u. \nabla \phi~dx - \la \int_\Om \frac{f(u)}{u^q}\phi~dx,\]
but since $\phi\in C_c^\infty(\Om)$ is arbitrary, Claim (2) follows. This completes the proof. \hfill{\QED}

\subsection{Sub and Supersolutions of $(P_\la)$}
We begin this section with the construction of our pair of sub and supersolutions and gradually prove our first main result, Theorem \ref{MT1}. The idea has been earlier used in \cite{KLS}. Let $e_1\in X$ denotes the first eigenfunction of $-\Delta_{p,w}$ which solves
\[-\Delta_{p,w}e_1 = \la_1 e_1^{p-1} \; \text{in}\; \Om, \;\; e_1=0\;\text{on}\; \partial \Om.\]
Then $e_1>0$, $e_1\in L^\infty(\Om)$, refer \cite{Drabek} and moreover, $e_1 \geq c_K>0$ on every $K \subset \subset \Om$ by Lemma \ref{Uniform}. By the hypothesis $(f_1)$ since $\lim\limits_{t\to 0} \frac{f(t)}{t^q}=\infty$, one can choose $a_\la>0$ sufficiently small such that
\[\la_1 (a_\la e_1)^{p-1} \leq \la \frac{f(a_\la e_1)}{(a_\la e_1)^q}.\]
Denoting by $\uline{u}=a_{\la}e_1$ we get
\[-\Delta_{p,w}\uline{u} \leq \la \frac{f(a_\la e_1)}{(a_\la e_1)^q}=\la \frac{f(\uline{u})}{\uline{u}^q}\;\text{in}\; \Om.\]
Now let $\oline{u}:= A_\la v_0$ where $0<v_0\in X\cap L^\infty(\Om)$ uniquely solves the problem
\[-\Delta_{p,w}v_0 = v_0^{-q},\; v_0>0 \; \text{in}\; \Om,\; v_0=0\;\text{on}\; \partial\Om,\]
for details, refer \cite{PG}. By the hypothesis $(f_1)$ since $\lim\limits_{t\to \infty}\frac{f(t)}{t^{q+p-1}}=0$, we choose $A_\la>0$ sufficiently large such that
\[\frac{f({A_\la \|v_0\|_\infty)}}{(A_\la \|v_0\|_\infty)^{q+p-1}}\leq \frac{1}{\la \|v_0\|^{q+p-1}_\infty}  \]
which gives
\[-\Delta_{p,w}\oline{u} = \frac{A_\la^{p-1}}{v_0^q}  \geq \la\frac{f({A_\la \|v_0\|_\infty)}}{(A_\la v_0)^{q}} \geq \la \frac{f(\oline{u})}{\oline{u}^q}\; \text{in}\; \Om \] where we have also used the non decreasing property of $f$ follows from $(f_1)$.
Therefore $\uline{u}$ and $\oline{u}$ forms sub and supersolution of $(P_\la)$ respectively and the constants $a_\la, A_\la$ can be chosen appropriately so that $\uline{u}\leq \oline{u}$.

\textbf{Proof of Theorem \ref{MT1}:} From above construction and using Lemma \ref{Subsuplemma}, we infer that $(P_\la)$ admits a weak solution $u\in X \cap L^\infty(\Om)$ such that $u \in [\uline{u}, \oline{u}]$. This proves Theorem \ref{MT1}. \hfill{\QED}

\section{Multiplicity result in Case (II)}
This section is devoted to prove our second main result that is Theorem \ref{MT3} using the method of approximation. We follow \cite{arcoya} here. Let us denote the energy functional $I_\la: X \to \mb R\cup \{\pm \infty\}$ corresponding to the problem $(P_\la)$ for Case (II)
\[I_\la(u) = \frac{1}{p}\int_\Om w(x)|\nabla u|^p~dx -\frac{\la}{1-q}\int_\Om (u^+)^{1-q}~dx -\frac{1}{r+1}\int_\Om (u^+)^{r+1}~dx.\]
For $\e>0$, let us consider the approximated problem
\begin{equation*}
(P_{\la,\e})\left\{
\begin{split}
  -\Delta_{p,w}u &=  \frac{\la}{(u^+ +\e)^q}+ (u^+)^r\;\text{in}\; \Om,\\
 u&=0 \; \text{on}\; \partial \Om
\end{split}\right.
\end{equation*}
for which the corresponding energy functional is given by
\[I_{\la,\e}(u) = \frac{1}{p}\int_\Om w(x)|\nabla u|^p~dx -\frac{\la}{1-q}\int_\Om [(u^+ +\e)^{1-q}-\e^{1-q}]~dx -\frac{1}{r+1}\int_\Om (u^+)^{r+1}~dx.\]
It is easy to verify that $I_{\la,\e}\in C^1(X,\mb R)$, $I_{\la,\e}(0)=0$ and $I_{\la,\e}(v)\leq I_{0,\e}(v)$ for all $v\geq 0.$ We recall the definition of $e_1$ from last section and w.l.o.g. assume that $\|e_1\|_\infty=1$. Our next Lemma states that $I_{\la,\e}$ satisfies the Mountain Pass geometry.

\begin{Lemma}\label{MP-geo}
There exists $R,\;\rho>0$ and $\Lambda>0$ depending on $R$ such that whenever $\la \in (0,\La)$
$$
\inf\limits_{\|v\|\leq R}I_{\la,\e}(v)<0\;\text{and}\;
\inf\limits_{\|v\|=R}I_{\la,\e}(v)\geq \rho.
$$
Moreover there exists $T>R$ such that 
$
I_{\la,\e}(Te_{1})<-1$ for $\la\in (0,\La)$.
\end{Lemma}
\begin{proof}
We fix $l=|\Om|^{\frac{1}{(\frac{p_s^*}{r+1})'}}$. Then using H\"{o}lder's inequality and Lemma \ref{embedding}, we get that
\begin{equation*}\label{MP1}
\int_\Om (v^+)^{r+1}~dx \leq \left( \int_\Om |v|^{p_s^*}\right)^{\frac{r+1}{p_s^*}} |\Om|^{\frac{1}{(\frac{p_s^*}{r+1})'}}\leq Cl||v||^{r+1}
\end{equation*}
for some positive constant $C$ independent of $v$. We now observe that
$$
\lim_{t\to 0}\frac{I_{\la,\e}(te_1)}{t}=-\la\e^{-q}\int_{\Omega}e_{1}\,dx<0,
$$
which implies that it is possible to choose $k\in(0,1)$ sufficiently small and to set $||v||=R :=k(\frac{r+1}{pCl})^\frac{1}{r+1-p}$ such that $\inf\limits_{\|v\|\leq R}I_{\la,\e}(v)<0.$
Moreover, since $R<(\frac{r+1}{pCl})^\frac{1}{r+1-p}$ we obtain
\begin{align*}
I_{0,\e}(v)\geq \frac{R^p}{p}-\frac{ClR^{r+1}}{r+1}&=\frac{k^p}{p}(\frac{r+1}{pCl})^\frac{p}{r+1-p}-\frac{Clk^{r+1}}{r+1}(\frac{r+1}{pCl})^\frac{r+1}{r+1-p}\\
& = \left(\frac{r+1}{pCl}\right)^{\frac{p}{r+1-p}}\left(\frac{k^p}{p}-\frac{Clk^{r+1}}{p}\right):=
2\rho\,(\text{say})>0.
\end{align*}
We define $$\Lambda:=\frac{\rho}{\sup\limits_{\|v\|=R} \left(\frac{1}{1-q}\int_\Om |v|^{1-q}~dx \right)}$$
which is a positive constant and since $\rho,R$ depends on $k,r,p,|\Omega|,C$ so does $\Lambda$. We know that
$$
((v^{+}+\e)^{1-q}-\e^{1-q})\leq (v^+)^{1-q}
$$
which gives
$$
I_{\lambda,\e}(v)\geq \frac{\|v\|^p}{p}-\frac{1}{r+1}\int_{\Om}(v^+)^{r+1}\,dx-\frac{\la}{1-q}\int_{\Om}(v^{+})^{1-q}\,dx\\
\geq I_{0,\e}(v)-\frac{\la}{1-q}\int_{\Om}(v^{+})^{1-q}\,dx.
$$
Therefore
$$
\inf\limits_{\|v\|=R} I_{\la,\e}(v)\geq\inf\limits_{\|v\|=R}I_{0,\e}(v)-\la \sup\limits_{\|v\|=R} \left(\frac{1}{1-q}\int_\Om |v|^{1-q}~dx \right)\geq 2\rho -\la \sup\limits_{\|v\|=R} \left(\frac{1}{1-q}\int_\Om |v|^{1-q}~dx \right)\geq \rho
$$
if $\la\in(0,\Lambda).$

Lastly, it is easy to see that $I_{0,\e}(te_1) \to -\infty$ as $t\to +\infty$ which implies that we can choose $T>R$ such that $I_{0,\e}(te_1)<-1$. Hence
\[I_{\la,\e}(Te_1)\leq I_{0,\e}(Te_1)<-1\]
which completes the proof.

                    \hfill{\QED}
\end{proof}\\

\noi As a consequence of Lemma \ref{MP-geo}, we have
\[\inf\limits_{\|v\|=R}I_{\la,\e}(v) \geq \rho \max\{I_{\la,\e}(te_1), I_{\la,\e}(0)\} = 0.\]
Our next Lemma ensures that $I_{\la,\e}$ satisfies the Palais Smale  $(PS)_c$ condition.

\begin{Proposition}\label{PS-cond}
$I_{\la,\e}$ satisfies the $(PS)_c$ condition, for any $c \in \mb R$ that is if $\{u_k\}\subset X$ is a sequence satisfying
\begin{equation}\label{PS1}
I_{\la,\e}(u_k)\to c \; \text{and}\; I_{\la,\e}^\prime(u_k) \to 0
\end{equation}
as $k \to \infty$ then $\{u_k\}$ contains a strongly convergent subsequence in $X$.
\end{Proposition}
\begin{proof}
Let $\{u_k\} \subset X$ satisfies \eqref{PS1} then we claim that $\{u_k\}$ must be bounded in $X$. To see this, we consider
\begin{equation}\label{PS2}
\begin{split}
I_{\la,\e}(u_k)- \frac{1}{r+1}I_{\la,\e}^\prime(u_k)u_k &= \left( \frac{1}{p}-\frac{1}{r+1}\right)\|u_k\|^p -\frac{\la}{1-q} \int_\Om [(u_k^+ +\e)^{1-q}-\e^{1-q}]~dx\\
& \quad +\frac{\la}{r+1}\int_\Om (u_k^+ +\e)^{-q}u_k~dx\\
& \geq \left( \frac{1}{p}-\frac{1}{r+1}\right)\|u_k\|^p -\frac{\la}{1-q}\int_\Om (u_k^+)^{1-q}~dx+\frac{\la}{r+1}\int_\Om (u_k^+ +\e)^{-q}u_k~dx\\
& \geq \left(\frac{1}{p}-\frac{1}{r+1}\right)\|u_k\|^p -C_1\int_\Om (u_k^+)^{1-q}~dx-C_{2}\e^{1-q}\\
& \geq C_3\|u_k\|^p -C_4 \|u_k\|^{1-q}-C_{2}\e^{1-q}
\end{split}
\end{equation}
where we have used the embedding theorems and $C_1,C_2, C_3, C_4>0$ are constants. Also from \eqref{PS1} it follows that for $k$ large enough
\begin{equation}\label{PS3}
\left| I_{\la,\e}(u_k)- \frac{1}{r+1}I_{\la,\e}^\prime(u_k)u_k\right| \leq c+o(\|u_k\|).
\end{equation}
Combining \eqref{PS2} and \eqref{PS3}, our claim follows. By reflexivity of $X$, we get that there exists a $u_0\in X$ such that up to a subsequence, $u_k \rightharpoonup u_0$ weakly in $X$ as $k \to \infty$.\\
\textbf{Claim:} $u_k \to u_0$ strongly in $X$ as $k \to \infty$.\\
By \eqref{PS1}, we already have that
\[\lim_{k\to \infty}\left(\int_\Om w(x)|\nabla u_k|^{p-2}\nabla u_k.\nabla u_0~dx - \la \int_\Om (u_k^+ +\e)^{-q}u_0~dx - \int_\Om (u_k^+)^{r} u_0~dx\right)=0\]
and
\[\lim_{k\to \infty}\left(\int_\Om w(x)|\nabla u_k|^{p-2}\nabla u_k.\nabla u_k~dx - \la \int_\Om (u_k^+ +\e)^{-q}u_k~dx - \int_\Om (u_k^+)^r u_k~dx\right)=0.\]
Now
\begin{equation}\label{PS4}
\begin{split}
&\lim\limits_{k\to\infty}\int_{\Omega}w(x)(|\nabla u_k|^{p-2}\nabla u_k-|\nabla u_0|^{p-2}\nabla u_0).\nabla(u_k-u_0)\,dx\\
&=\lim\limits_{k\to\infty} \left( \la \int_\Om (u_k^+ +\e)^{-q}u_k~dx + \int_\Om (u_k^+)^r u_k~dx - \la \int_\Om (u_k^+ +\e)^{-q}u_0~dx - \int_\Om (u_k^+)^r u_0~dx\right)\\
&\quad -\lim_{k\to \infty}\left(\int_\Om w(x)|\nabla u_0|^{p-2}\nabla u_0. \nabla u_k~dx - \int_\Om w(x) |\nabla u_0|^p~dx\right).
\end{split}
\end{equation}
From weak convergence of $\{u_k\}$ we get
\begin{equation}\label{PS5}
\lim_{k\to \infty}\left(\int_\Om w(x)|\nabla u_0|^{p-2}\nabla u_0. \nabla u_k~dx - \int_\Om w(x) |\nabla u_0|^p~dx\right)=0.
\end{equation}
Also $|(u_k^++\e)^{-q}u_0| \leq\e^{-q}u_0$ and Lebesgue Dominated convergence theorem gives that
\begin{equation}\label{PS6}
\lim_{k \to \infty} \int_\Om (u_k^+ +\e)^{-q}u_0~dx = \int_\Om (u_0^+ +\e)^{-q}u_0~dx.
\end{equation}
 Since $u_k \to u_0$ a.e. in $\Om$ and for any measurable subset $E$ of $\Om$ we have
\[\int_E |(u_k^++\e)^{-q}u_k |~dx \leq \int_E\e^{-q}u_k~dx \leq C_1\|u_k\|_{L^{p_s^*}(\Om)}|E|^{\frac{p_s^*-1}{p_s^*}}\leq C_2|E|^{\frac{p_s^*-1}{p_s^*}}, \]
so from Vitali convergence theorem it follows that
\begin{equation}\label{PS7}
\lim\limits_{k\to\infty} \la \int_\Om (u_k^+ +\e)^{-q}u_k~dx = \la \int_\Om (u_0^+ +\e)^{-q}u_0~dx .
\end{equation}
Similarly, we have
\[\int_E |(u_k^+)^ru_0|~dx \leq \|u_0\|_{L^{p_s^*}(\Om)} \left(\int_E (u_k^+)^{rp_s^{*'}}~dx\right)^{\frac{1}{p^{*'}_s}}\leq C_3 |E|^{\alpha}  \]
and
\[\int_E |(u_k^+)^ru_k|~dx \leq \|u_k\|_{L^{p_s^*}(\Om)} \left(\int_E (u_k^+)^{rp_s^{*'}}~dx\right)^{\frac{1}{p^{*'}_s}}\leq C_3 |E|^{\beta}  \]
for some constants $\alpha>0,\beta>0$ which using Vitali convergence theorem  implies that
\begin{equation}\label{PS8}
\lim_{k \to \infty} \int_\Om (u_k^+)^ru_0~dx  =\int_\Om (u_0^+)^ru_0~dx \text{ and } \lim_{k \to \infty} \int_\Om (u_k^+)^ru_k~dx  =\int_\Om (u_0^+)^ru_0~dx.
\end{equation}
Putting \eqref{PS5}, \eqref{PS6}, \eqref{PS7} and \eqref{PS8} in \eqref{PS4} we obtain
\[\lim\limits_{k\to\infty}\int_{\Omega}w(x)(|\nabla u_k|^{p-2}\nabla u_k-|\nabla u_0|^{p-2}\nabla u_0).\nabla(u_k-u_0)\,dx =0.\]
From \cite{PG}, we know that
\begin{align*}
&\int_{\Omega}w(x)(|\nabla u_k|^{p-2}\nabla u_k-|\nabla u_0|^{p-2}\nabla u_0).\nabla(u_k-u_0)\,dx\\
& \quad \quad \quad \geq (\|u_k\|^{p-1}-\|u_0\|^{p-1})(\|u_k\|-\|u_0\|)
\end{align*}
which proves our claim. \hfill{\QED}
\end{proof}

From Lemma \ref{MP-geo}, Proposition \ref{PS-cond} and Mountain Pass Lemma, we get that there exists a $\zeta_\e \in X$ such that $I_{\lambda,\e}^\prime(\zeta_\e)=0$ such that
$$
I_{\lambda,\e}(\zeta_{\e})=\inf_{\gamma\in\Gamma}\max_{t \in [0,1]}I_{\la,\e}(\gamma (t)) \geq \rho >0
$$
where $\Gamma = \{\gamma \in C([0,1];X): \gamma(0)=0, \gamma(1))=Te_1\}.$ Furthermore, as a consequence of Lemma \ref{MP-geo}, since $\inf\limits_{\|v\|\leq R} I_{\la,\e}(v)<0$, from weak lower semicontinuity of the functional $I_{\la,\e}$ we get that there exists $\nu_\e \not\equiv 0$ such that $\|\nu_\e\| \leq R$ and
\begin{equation}\label{limit-pass}
\inf\limits_{\|v\|\leq R} I_{\la,\e}(v) =I_{\la,\e}(\nu_\e)<0 < \rho \leq I_{\la,\e}(\zeta_\e).
\end{equation}
Thus, $\zeta_\e$ and $\nu_\e$ are two different non trivial critical points of $I_{\la,\e}$. Testing $(P_{\la,\e})$ with $\min\{\zeta_\e,0\}$ and $\min\{\nu_\e,0\}$, it is easy to verify that $\zeta_\e,\nu_\e\geq 0$ since the R.H.S. of $(P_{\la,\e})$ remains a non negative quantity.

\begin{Lemma}\label{apriori}
There exists a $\Theta>0$ (independent of $\e$) such that $\|v_\e\| \leq \Theta$ where $v_\e = \zeta_\e$ or $\nu_\e$.
\end{Lemma}
\begin{proof}
The result trivially holds if $v_\e = \nu_\e$ so we deal with the case $v_\e= \zeta_\e$. Recalling the terms from Lemma \ref{MP-geo}, we define $A = \max\limits_{t \in [0,1]}I_{0,\e}(tTe_1)$ then
\[A \geq \max_{t \in [0,1]} I_{\la,\e}(tTe_1) \geq\inf_{\gamma\in\Gamma}\max_{t \in [0,1]}I_{\la,\e}(\gamma (t)) = I_{\la,\e}(\zeta_\e).\]
Therefore
\begin{equation}\label{ap1}
\frac{1}{p}\int_\Om w(x)|\nabla \zeta_{\e}|^p~dx -\frac{\la}{1-q}\int_\Om [(\zeta_{\e}+\e)^{1-q}-\e^{1-q}]~dx -\frac{1}{r+1}\int_\Om \zeta_{\e}^{r+1}~dx \leq A.
\end{equation}
Choosing $\phi=-\frac{\zeta_{\e}}{r+1}$ as a test function in $(P_{\la,\e})$ we obtain
\begin{equation}\label{ap2}
-\frac{1}{r+1}\int_{\Om}w(x)|\nabla \zeta_{\e}|^{p}\,dx+\frac{\la}{r+1}\int_{\Om}\frac{\zeta_{\e}}{(\zeta_{\e}+\e)^q}\,dx+\frac{1}{r+1}\int_{\Om}\zeta_{\e}^{r+1}\,dx=0.
\end{equation}
Adding \eqref{ap1} and \eqref{ap2} we get
\begin{align*}
\left(\frac{1}{p}-\frac{1}{r+1}\right)\int_{\Om}w(x)|\nabla \zeta_{\e}|^p\,dx
 &\leq \frac{\la}{1-q}\int_\Om [(\zeta_{\e} +\e)^{1-q}-\e^{1-q}]~dx -\frac{\la}{r+1}\int_{\Om}\frac{\zeta_{\e}}{(\zeta_{\e}+\e)^q}\,dx+A\\
 & \leq  \frac{\la}{1-q}\int_\Om [(\zeta_{\e} +\e)^{1-q}-\e^{1-q}]~dx +A\\
 & \leq  \frac{\la}{1-q}\int_\Om \zeta_{\e}^{1-q}~dx +A\leq C \|\zeta_{\e}\|^{1-q}+A,
\end{align*}
where we have used H\"older inequality along with the embedding result Lemma \ref{embedding} and $C>0$ is a constant independent of $\e$. This implies that $\{\zeta_{\e}\}$ is uniformly bounded in $X$ with respect to $\e$. This completes the proof.\hfill{\QED}
\end{proof}\\

Now as a resultant of Lemma \ref{apriori}, up to a subsequence we get that $\zeta_\e \rightharpoonup \zeta_0$ and $\nu_\e \rightharpoonup \nu_0$ weakly in $X$ as $\e \to 0^+$, for some non negative $\zeta_0,\nu_0\in X$.  In the sequel, we establish that $\zeta_0\neq \nu_0$ and forms a weak solution to our problem $(P_\la)$.
For convenience we denote by $v_0$ either $\zeta_0$ or $\nu_0$.
\begin{Lemma}\label{Solution}
$v_0\in X$ is a weak solution to the problem $(P_{\la})$.
\end{Lemma}
\begin{proof}
We observe that for any $\e\in(0,1)$ and $t\geq 0$
$$
\frac{\la}{(t+\e)^q}+t^r\geq \frac{\la}{(t+1)^q}+t^r\geq \text{min}\{1,\frac{\la}{2^q}\}.
$$
As a consequence we get
$$
-\Delta_{p,w} v_\e=\frac{\la}{(v_\e+\e)^q}+v_\e^r\geq \text{min}\{1,\frac{\la}{2^q}\}:=C,\text{say}.
$$
Consequently, if $\xi\in X$ satisfies
$$
-\Delta_{p,w}\xi =C\text{ in }\Om
$$
we get
\begin{equation}\label{strict positivity}
\int_{\Om}w(x)|\nabla v_\e|^{p-2}\nabla v_\e.\nabla \phi\,dx\geq \int_{\Om} w(x)|\nabla \xi|^{p-2}\nabla \xi.\nabla \phi\,dx
\end{equation}
for every non negative $\phi\in X$. Therefore choosing $\phi=(\xi- v_\e)^{+}\in X$ as a test function in \eqref{strict positivity} we obtain using algebraic inequality Lemma \ref{alg-ineq} that
$$
v_\e\geq \xi\text{ in }\Om.
$$
Now by the Strong maximum principle (see \cite{Juh}) we obtain $\xi>0$ in $\Om$. Now by Lemma \ref{Uniform} we obtain that $\xi\geq c_K>0$ for every $K\subset\subset\Om$. Therefore
\begin{equation}\label{uniform}
v_\e\geq c_K>0
\end{equation}
for every $K\subset\subset\Om$. Therefore using Lemma \ref{apriori} and the fact \eqref{uniform} we can apply Theorem 2.20 of \cite{PG} to pass the limit and obtain
$$
\int_{\Om}w(x)|\nabla v_0|^{p-2}\nabla v_0.\nabla \phi\,dx=\la\int_{\Om}\frac{\phi}{v_0^q}\,dx+\int_{\Om}v_0^{r}\phi\,dx.
$$
This completes the proof.
\end{proof}

\noi \textbf{Proof of Theorem \ref{MT3}:}
Using Lemma \ref{Solution} we get that $\zeta_0$ and $\nu_0$ are two positive weak solution of $(P_{\la})$. Now we are going to prove that $\zeta_0\neq \nu_0$.  Choosing $\phi=v_\e\in X$ as a test function in $(P_{\la,\e})$ we get
$$
\int_{\Omega}w(x)|\nabla v_\e|^p\,dx=\la\int_{\Om}\frac{v_\e}{(v_\e+\e)^q\,dx}+\int_{\Om}(v_\e)^{r+1}\,dx
$$
Since $r+1<p_s^{*}$, using Lemma \ref{embedding} we obtain
$$
\lim\limits_{\e\to 0}\int_{\Om}(v_\e)^{r+1}\,dx=\int_{\Om}v_0^{r+1}\,dx.
$$
Moreover, since
$$
0\leq \frac{v_\e}{(v_\e+\e)^q}\leq v_\e^{1-q},
$$
by Vitali convergence theorem
$$
\la\lim\limits_{\e\to 0}\int_{\Om}\frac{v_\e}{(v_\e+\e)^q}\,dx=\la\int_{\Om}(v_0)^{1-q}\,dx.
$$
Therefore
$$
\lim\limits_{\e\to 0}\int_{\Om}w(x)|\nabla v_\e|^{p}\,dx=\la\int_{\Om}(v_0)^{1-q}\,dx+\int_{\Om}(v_0)^{r+1}\,dx.
$$
Using Lemma \ref{testfn} we can choose $\phi=v_0$ as  a test function in $(P_\la)$ to deduce that
$$
\int_{\Om}w(x)|\nabla v_0|^{p}\,dx=\la\int_{\Om}(v_0)^{1-q}\,dx+\int_{\Om}(v_0)^{r+1}\,dx.
$$
Hence we obtain
$$
\lim\limits_{\e\to 0}\int_{\Om}w(x)|\nabla v_\e|^{p}\,dx=\int_{\Om}w(x)|\nabla v_0|^{p}\,dx
$$
and we get the strong convergence of $v_\e$ to $v_0$ in $X$.
Now by the Lebesgue dominated theorem, we get
$$
\lim\limits_{\e\to 0}\int_{\Om}[(v_\e+\e)^{1-q}-\e^{1-q}]\,dx=\int_{\Om}(v_0)^{1-q}\,dx,
$$
which together with the strong convergence of $v_\e$ implies
$
\lim\limits_{\e\to 0}I_{\la,\e}(v_\e)=I_{\la}(v_0).
$
Hence from \eqref{limit-pass} we get $\zeta_0\neq \nu_0.$ \hfill{\QED}

\section*{Acknowledgments}
We thank T.I.F.R. CAM-Bangalore for the financial support.

\bibliographystyle{plain}
\bibliography{mybibfile}

\end{document}